\documentclass[12pt]{amsart}
\usepackage{fullpage}

\newcommand{\g}{\mathfrak{g}}
\newcommand{\h}{\mathfrak{h}}
\newcommand{\mfk}{\mathfrak{k}}
\newcommand{\mfp}{\mathfrak{p}}
\newcommand{\mft}{\mathfrak{t}}
\newcommand{\mfb}{\mathfrak{b}}
\newcommand{\mfa}{\mathfrak{a}}
\newcommand{\mfn}{\mathfrak{n}}
\newcommand{\bbZ}{\mathbb{Z}}
\newcommand{\bbC}{\mathbb{C}}

\newcommand{\cC}{\mathcal{C}}
\newcommand{\cS}{\mathcal{S}}

\DeclareMathOperator{\sgn}{\mathrm{sgn}}

\newtheorem{theorem}{Theorem}[section]
\newtheorem{proposition}[theorem]{Proposition}
\newtheorem{lemma}[theorem]{Lemma}

\theoremstyle{definition}
\newtheorem{definition}[theorem]{Definition}
\newtheorem{notation}[theorem]{Notation}

\theoremstyle{remark}


\title[Signed and Classical Kazhdan-Lusztig Polynomials]
{Relating Signed Kazhdan-Lusztig Polynomials and Classical 
Kazhdan-Lusztig Polynomials}

\author{Wai~Ling~Yee}
\address{Department of Mathematics and Statistics \\ 
University of Windsor \\
Windsor, Ontario \\
CANADA}
\email{wlyee@uwindsor.ca}
\thanks{The author is grateful for the support from a Discovery 
Grant and UFA from NSERC, and NSF grant DMS-0554278.}
\dedicatory{This paper is dedicated to the memory of Amir Hussain.}

\begin{document}

\begin{abstract}
Motivated by studying the Unitary Dual Problem, a variation of 
Kazhdan-Lusztig polynomials was defined in \cite{Y} which encodes signature 
information at each level of the Jantzen filtration.  These so called 
signed Kazhdan-Lusztig polynomials may be used to compute the signatures of 
invariant Hermitian forms on irreducible highest weight modules.
The key result of this paper is a simple relationship between signed 
Kazhdan-Lusztig polynomials and classical Kazhdan-Lusztig polynomials:
signed Kahzdan-Lusztig polynomials are shown to equal classical Kazhdan-Lusztig 
polynomials evaluated at $-q$ rather than $q$ and multiplied by a sign.  This 
result has applications to finding the unitary dual for real reductive Lie 
groups since Harish-Chandra modules may be constructed by applying Zuckerman 
functors to highest weight modules.
\end{abstract}

\maketitle

\section{Introduction}
Classifying irreducible unitary representations of a group, known as 
the Unitary Dual Problem, is an open problem that is important for its wide
ranging applications.  In particular, it is a necessary component of a
programme in abstract harmonic analysis articulated by I.M. Gelfand in the 
1930s for solving difficult problems in disparate areas of mathematics.  
Gelfand's general philosophy is to formulate the solution as the solution to a 
corresponding algebraic problem which, in turn, may be solved by decomposition 
into simpler (though possibly infinitely many) problems.

The most general approach towards solving the Unitary Dual Problem for real 
reductive Lie groups has been to first identify a broader set of 
representations: the Hermitian representations, which accept invariant 
Hermitian forms.  By computing the signatures of these invariant Hermitian 
forms and then identifying when the forms are definite, one obtains a 
classification of the unitary representations.  The cases for which the 
Unitary Dual Problem is solved are limited.

In \cite{Y}, signed Kazhdan-Lusztig polynomials for semisimple Lie algebras 
were introduced and used to give formulas for signature characters of invariant 
Hermitian forms on irreducible highest weight modules.  While unitary highest 
weight modules have been identified by the work of Enright-Howe-Wallach, 
understanding signatures of all Hermitian representations is important since 
Harish-Chandra modules may be constructed by applying Zuckerman or Bernstein 
functors to highest weight modules.  Identifying the irreducible unitary 
represenations of a real reductive Lie group is equivalent to classifying 
irreducible Harish-Chandra modules.  While the Zuckerman functor is known to 
preserve unitarity in certain circumstances (\cite{V},\cite{W}), it does not 
preserve unitarity in general, hence the need to understand signatures of 
all Hermitian highest weight modules.  This paper and \cite{Y2} provide 
dramatic simplifications to the formulas in \cite{Y} for signed 
Kazhdan-Lusztig polynomials and signatures of invariant Hermitian forms on 
irreducible highest weight modules in the equal rank case.  Amazingly, 
signed Kazhdan-Lusztig polynomials are equal to classical Kazhdan-Lusztig 
polynomials evaluated at $-q$ up to a sign.  Specifically:

{\bf Main Theorem:}
Let $\g_0$ be a real equal rank semisimple Lie algebra with complexification 
$\g$, $\theta$ a Cartan involution of $\g_0$, and let $\h_0$ be a 
$\theta$-stable Cartan subalgebra with complexification $\h$.  Let $\lambda  
\in \h^*$ be antidominant and let $x$ be in the integral Weyl group of 
$\lambda$ such that the Verma module $M(x \lambda )$ admits an invariant 
Hermitian form (more details within the paper).  
Then:
$$P_{x,y}^{\lambda,w_0}(q) =  (-1)^{\epsilon( x \lambda - y \lambda )}
P_{x,y}(-q)$$
where $\epsilon$ is the $\bbZ_2$-grading on the imaginary root lattice.  (That 
is, $\epsilon( \mu )$ is the parity of the number of non-compact roots in an 
expression for $\mu$ as an integral linear combination of roots.)  The 
polynomial on the left hand side is a signed Kazhdan-Lusztig polynomial while 
the polynomial on the right hand side is a classical Kazhdan-Lusztig 
polynomial.

The format of the paper is as follows.

Sections 2 and 3 contain a synopsis of signature character theory for Verma 
modules and for irreducible highest weight modules.

In section 4, we simplify the formulas for the signs 
$\varepsilon$ which appear in the formulas in Sections 2 and 3.

Section 5 contains the proof of the main theorem.

In Section 6, we discuss upcoming work.

\section{Signature Characters for Invariant Hermitian Forms on Irreducible 
Verma Modules}
In this section, we will restrict our attention to the equal rank case although 
more general formulas appeared in \cite{Y3}.  This streamlines the exposition 
as it eliminates the additional complications which arise in the non-equal 
rank case.
\begin{notation}
We use the following notation in this section:
\begin{itemize}
 \item[-] $\g_0$ is a real equal rank semisimple Lie algebra
 \item[-] $\theta$ is a Cartan involution on $\g_0$ inducing the decomposition 
$\g_0 = \mfk_0 \oplus \mfp_0$
 \item[-] $\h_0 = \mft_0 \oplus \mfa_0$ is the Cartan decomposition of a 
$\theta$-stable Cartan subalgebra
 \item[-] omitting the subscript $0$ indicates complexification
 \item[-] $\mfb = \h \oplus \mfn$ is a Borel subalgebra giving positive roots
	  $\Delta^+( \g, \h)$ and $\g = \mfn \oplus \h \oplus \mfn^-$ is the 
	  corresponding triangular decomposition
 \item[-] $\Lambda_r$ is the root lattice
 \item[-] $\rho$ is one half the sum of the positive roots
 \item[-] $\alpha_1, \ldots, \alpha_n$ are the simple roots
	  and $s_1, \ldots, s_n$ the corresponding simple reflections
 \item[-] $W$ is the Weyl group and $\mathfrak{C}_0$, the fundamental 
 	  chamber, is chosen to be antidominant
 \item[-] $\lambda \in \h^*$
 \item[-] $M(\lambda ) = U(\g) \otimes_{U(\mfb)} \bbC_{\lambda-\rho}$ is 
	  the Verma module of highest weight $\lambda - \rho$ with canonical 
	  generator $v_{\lambda-\rho}$
 \item[-] $\bar{\cdot}$ applied to elements of $\g$ and $\h^*$ denotes complex 
	  conjugation relative to the real form $\g_0$
 \item[-] $H_{\alpha,N}$ denotes the affine hyperplane $H_{\alpha,N} = 
	\{ \lambda \in \h^* : (\lambda,\alpha^\vee) = n \}$
	  where $\alpha \in \Delta(\g,\h)$ and $n \in \bbZ$.
 \item[-] $W_a$ is the affine Weyl group with fundamental (antidominant) 
	  alcove $A_0$
\end{itemize}
\end{notation}

\begin{definition}
An invariant Hermitian form on a representation $V$ is a sesquilinear pairing 
$\left< \cdot, \cdot \right> :  V \times V \to \bbC$
such that
$$\left< X \cdot v, w \right> + \left< v, \bar{X}\cdot w \right> = 0$$
for every $v, w \in V$ and $X \in \g$.
\end{definition}

\begin{proposition} (\cite{Y3}, p. 641)
The Verma module $M( \lambda )$ admits a non-trivial invariant Hermitian 
form if and only if $\theta(\mfb) = \mfb$ and $\lambda-\rho$ is imaginary.
\end{proposition}

Such a form is unique up to a real scalar.  If $\left< v_{\lambda-\rho},
v_{\lambda-\rho} \right> = 1$, then the form is called the Shapovalov form 
and is denoted by $\left< \cdot, \cdot \right>_{\lambda}$.  Henceforth, for 
the remainder of this paper, we fix $\mfb$ and $\lambda$ to satisfy these 
conditions.

Because of invariance, the weight space decomposition of the Verma module is 
an orthogonal decomposition, making the notion of the signature of the 
Shapovalov form reasonable although the Verma module is infinite-dimensional.  
Thus:
\begin{definition}
If the Shapovalov form $\left< \cdot, \cdot \right>_\lambda$ is 
non-degenerate and has signature 
$(p(\mu), q(\mu))$ on the $\lambda-\mu-\rho$ weight space,  the signature 
character of the Shapovalov form on $M(\lambda)$ is:
$$ch_s M(\lambda) = \sum_{\mu \in \Lambda_r^+}( p(\mu) - q(\mu)) e^{\lambda
- \mu-\rho}.$$
\end{definition}

The radical of the Shapovalov form is precisely the unique maximal proper 
submodule of $M(\lambda)$.  The Shapovalov determinant formula states the 
following:

\begin{proposition}
The determinant of a matrix representing the Shapovalov form on the $\lambda
-\rho - \mu$ weight space, up to a constant, is
$$\prod_{\alpha \in \Delta^+(\g,\h)} \prod_{n=1}^\infty 
((\lambda, \alpha^\vee) - n)^{P(\mu - n \alpha)}$$
where $P$ denotes the Kostant partition function.
\end{proposition}
Thus the Shapovalov form is degenerate precisely on the reducibility 
hyperplanes $H_{\alpha,n}$ where $\alpha$ and $n$ are positive.  In any 
region bounded by these hyperplanes, the Shapovalov form stays non-degenerate, 
whence the signature remains constant.

In \cite{W}, Nolan Wallach used an asymptotic argument to determine the 
signature on the largest of these regions:
\begin{theorem} (\cite{W}, Lemma 2.3)
Let $\lambda$ be imaginary and $\mfb$ be $\theta$-stable.  If 
$(\lambda, \alpha^\vee) < 1$ for every positive root $\alpha$, then
$$ch_s M( \lambda ) = 
\frac{e^{\lambda-\rho}}
{\prod_{\alpha \in \Delta^+(\mfk,\h)} ( 1 + e^{-\alpha}) \prod_{\alpha \in 
\Delta^+(\mfp, \h)}(1 - e^{-\alpha})}.$$
\end{theorem}
As a historical note, Wallach showed that the Zuckerman functor applied to 
these modules produced unitary representations.

Formulas for signature characters for all irreducible Verma modules admitting 
invariant Hermitian forms may be found in \cite{Y3}.  The proof uses 
the following philosophy found in \cite{V}.  Within any region bounded by 
reducibility hyperplanes, the signature remains constant.  The goal is to 
understand how signatures change as you cross a reducibility hyperplane into 
another region.  If you cross only one reduciblity hyperplane at a time, the 
structure of the corresponding Jantzen filtration (see \ref{JantzenFiltration}
for the definition)
is simple and the signature changes by the signature of the radical, which is 
also a Verma module:

\begin{lemma} \label{WallCrossing} ( \cite{Y3}, Proposition 3.2)
Suppose $\lambda_t : (-\delta,\delta) \to \h^*$ is a path for which every 
$M(\lambda_t)$ admits a non-degenerate invariant Hermitian form.  Suppose 
$M(\lambda_t)$ is irreducible for $t\neq0$ and $\lambda_0$ belongs to the 
reducibility hyperplane $H_{\alpha,n}$ but not to any other reducibility 
hyperplane.  Suppose for $t \in (0, \delta)$ that is in the positive half 
space $\lambda_t \in H_{\alpha,n}^+$ while for $t \in (-\delta,0)$, 
$\lambda_t \in H_{\alpha,n}^-$.
Let $t_1 \in (0, \delta)$ and let $t_2 \in (-\delta, 0)$.
Then:
$$ch_s M(\lambda_{t_1}) = e^{\lambda_{t_1}-\lambda_{t_2}} ch_s M( \lambda_{t_2})
+ 2 \varepsilon( H_{\alpha,n}, \lambda_0) ch_s M( \lambda_{t_1} - n \alpha ).$$
for some sign $\varepsilon( H_{\alpha,n}, \lambda_0 ) = \pm 1$.
We can extend the definition of $\varepsilon$ to other affine hyperplanes by 
setting $\varepsilon( H_{\alpha,n}, \lambda_0)=0$
when the only affine hyperplane $\lambda_0$ belongs to is $H_{\alpha,n}$ and 
$H_{\alpha,n}$ is not a reducibility hyperplane.
\end{lemma}

It turns out that $\varepsilon( H_{\alpha,n}, \lambda_0 )$ stays constant 
over $\lambda_0$ in a given Weyl chamber, so we let $\varepsilon( H_{\alpha,n},
s )$ be that value in the Weyl chamber $s \mathfrak{C}_0$.

$\varepsilon( H_{\alpha,n}, s )$ is computed in \cite{Y3}:

\begin{theorem}\label{wgammadefn}( \cite{Y3} Theorems 6.1.12, 5.2.18)
Let $\gamma$ be a positive root and let $\gamma = s_{i_1} \cdots s_{i_{k-1}}
\alpha_k $ be such that $ht( s_{i_j} \cdots s_{i_{k-1}} \alpha_{i_k})$
decreases as $j$ increases.  Let $w_{\gamma} = s_{i_1} \cdots  s_{i_k}$.
If $\gamma$ hyperplanes are reducibility hyperplanes on $s \mathfrak{C}_0$
and if $\gamma$ does not form a type $G_2$ root system with other roots, 
then:
\begin{eqnarray*}
\varepsilon(H_{\gamma,N}, s ) = & &(-1)^{N \# \{ \text{noncompact } \alpha_{i_j}
	: |\alpha_{i_j}| \geq |\gamma| \}} \\
&\times & (-1)^{\#\{\beta \in \Delta( w_\gamma^{-1}) : |\beta| = |\gamma|,
		\beta \neq \gamma, \text{ and } \beta, s_\beta \gamma \in
		\Delta(s^{-1}) \}} \\
&\times & (-1)^{\#\{\beta \in \Delta( w_\gamma^{-1}) : |\beta| \neq |\gamma| 
		\text{ and } \beta, -s_\beta s_\gamma \beta \in \Delta(s^{-1})
		\}}.
\end{eqnarray*}
Let $\alpha_1$ and $\alpha_2$ be the long and short simple roots for a type 
$G_2$ root system, respectively.  Let $\delta_\alpha = 1$ if $\alpha$ is 
compact, and let it be $-1$ if $\alpha$ is non-compact.  We have:
\begin{equation*}
\begin{array}{|c|c|c|c|c|c|c|}
\hline
\text{{\bf Hyperplane}} & \multicolumn{6}{c|}{\text{\bf{Weyl Chamber }}
s\mathfrak{C}_0} \\ \hline \hline
H_{\alpha_1,N} & 
s_1 & s_1 s_2 & s_1 s_2 s_1 &
s_1 s_2 s_1 s_2 & s_1 s_2 s_1 s_2 s_1 & s_1 s_2 s_1 s_2 s_1 s_2  \\ \cline{2-7}
& \delta_{\alpha_1}^N & \delta_{\alpha_1}^N & \delta_{\alpha_1}^N &
\delta_{\alpha_1}^N & \delta_{\alpha_1}^N & \delta_{\alpha_1}^N  \\ \hline
\hline
H_{\alpha_1+\alpha_2,N} & 
s_1 s_2 & s_1 s_2 s_1 & s_1 s_2 s_1 s_2 &
s_1 s_2 s_1 s_2 s_1 & s_1 s_2 s_1 s_2 s_1 s_2 &  s_2 s_1 s_2 s_1 s_2 \\ \cline{2-7}
& \delta_{\alpha_1+\alpha_2}^N &
-\delta_{\alpha_1+\alpha_2}^N &
-\delta_{\alpha_1+\alpha_2}^N &
-\delta_{\alpha_1+\alpha_2}^N &
-\delta_{\alpha_1+\alpha_2}^N &
\delta_{\alpha_1+\alpha_2}^N  \\ \hline \hline
H_{2\alpha_1+3\alpha_2,N} & 
s_1 s_2 s_1 & s_1 s_2 s_1 s_2 & s_1 s_2 s_1 s_2 s_1 & 
s_1 s_2 s_1 s_2 s_1 s_2 & s_2 s_1 s_2 s_1 s_2 & s_2 s_1 s_2 s_1 \\ \cline{2-7}
& \delta_{2\alpha_1+3\alpha_2}^N & -\delta_{2\alpha_1+3\alpha_2}^N & 
\delta_{2\alpha_1+3\alpha_2}^N & \delta_{2\alpha_1+3\alpha_2}^N & 
-\delta_{2\alpha_1+3\alpha_2}^N & \delta_{2\alpha_1+3\alpha_2}^N \\
 \hline \hline

H_{\alpha_1+2\alpha_2,N} & 
s_1 s_2 s_1 s_2 & s_1 s_2 s_1 s_2 s_1 & s_1 s_2 s_1 s_2 s_1 s_2 & 
s_2 s_1 s_2 s_1 s_2 & s_2 s_1 s_2 s_1 & s_2 s_1 s_2 \\ \cline{2-7}
& \delta_{\alpha_1+2\alpha_2}^N & -\delta_{\alpha_1+2\alpha_2}^N & 
\delta_{\alpha_1+2\alpha_2}^N & 
\delta_{\alpha_1+2\alpha_2}^N & -\delta_{\alpha_1+2\alpha_2}^N & 
\delta_{\alpha_1+2\alpha_2}^N  
\\ \hline \hline
H_{\alpha_1+3\alpha_2,N} & 
s_1 s_2 s_1 s_2 s_1 & s_1 s_2 s_1 s_2 s_1 s_2 & s_2 s_1 s_2 s_1 s_2 & 
s_2 s_1 s_2 s_1 & s_2 s_1 s_2 & s_2 s_1 \\ \cline{2-7}
& \delta_{\alpha_1+3\alpha_2}^N & 
-\delta_{\alpha_1+3\alpha_2}^N &
-\delta_{\alpha_1+3\alpha_2}^N &
-\delta_{\alpha_1+3\alpha_2}^N &
-\delta_{\alpha_1+3\alpha_2}^N &
\delta_{\alpha_1+3\alpha_2}^N \\ \hline \hline
H_{\alpha_2,N} & 
s_1 s_2 s_1 s_2 s_1 s_2 & s_2 s_1 s_2 s_1 s_2 & s_2 s_1 s_2 s_1 &
s_2 s_1 s_2 & s_2 s_1 & s_2 \\ \cline{2-7}
& \delta_{\alpha_2}^N & \delta_{\alpha_2}^N & \delta_{\alpha_2}^N &
\delta_{\alpha_2}^N & \delta_{\alpha_2}^N & \delta_{\alpha_2}^N
\\ \hline
\end{array} 
\end{equation*}
\end{theorem}

Since alcoves defined by the action of the affine Weyl group are bounded by 
affine 
hyperplanes of the form $H_{\alpha,n}$, where $\alpha$ is a root and $n \in 
\bbZ$, the signature within the interior of an alcove is constant and it 
makes sense to define $\varepsilon( A, A')$ for adjacent alcoves $A, A'$ 
separated by the affine hyperplane $H_{\alpha,n}$ where for $\lambda \in A$
and $\lambda' \in A'$:
$$ch_s \, M( \lambda ) = e^{\lambda-\lambda'} ch_s \, M( \lambda' ) + 2 
\varepsilon( A, A' ) ch_s \, M( \lambda - n\alpha ).$$
Observe that $\varepsilon( A, A' ) = - \varepsilon( A', A )$.

Given our formula for $\varepsilon(H_{\alpha,n},s)$, we know how signatures 
change as we cross reducibility hyperplanes.  We cross one reducibility 
hyperplane at a time until we reach the region where signatures are known 
by Wallach's work.  We find by induction:
\begin{theorem} (\cite{Y3} Theorem 4.6)
Let: 
\begin{itemize}
\item[-] $\bar{\cdot} : W_a \to W$ denote the group homomorphism defined by 
sending $w = ts$ to $s$ where $t$ is translation by an element of the root 
lattice and $s$ is an element of $W$
\item[-] $\tilde{\cdot} : W_a \to W$ be defined by sending $w$ to $\tilde{w}$ 
where $\tilde{w} \mathfrak{C}_0$ is the Weyl chamber containing $w A_0$
\item[-] $\lambda \in w A_0$ such that $M(\lambda )$ admits a nondegenerate 
	invariant Hermitian form
\item[-]  $\displaystyle{R(\mu ) = \frac{e^{\mu-\rho}}
{\prod_{\alpha \in \Delta^+(\mfk,\h)} ( 1 + e^{-\alpha}) \prod_{\alpha \in 
\Delta^+(\mfp, \h)}(1 - e^{-\alpha})}}$
\item[-] $w A_0 = C_0 \xrightarrow{r_1} C_1 \xrightarrow{r_2} \cdots 
	\xrightarrow{r_\ell} C_\ell = \tilde{w} A_0$ a path from $w A_0$ to 
	$\tilde{w} A_0$ where the $C_i$'s are the alcoves on the path and 
	the $r_i$'s are the reflections through the affine hyperplanes 
	separating the alcoves
\end{itemize}
Then
$$ ch_s M(\lambda ) = \sum_{I = \{ i_1 < \cdots < i_k \} \subset \{1,\ldots,\ell
	\}} \varepsilon( I ) 2^{|I|} R( \overline{r_{i_1}r_{i_2} \cdots 
r_{i_k}} r_{i_k} \cdots r_{i_1} \lambda )
$$
where $\varepsilon( \{\} ) = 1$ and 
$\varepsilon( I ) = \varepsilon( C_{i_1 -1}, C_{i_1}) \varepsilon( \overline{
r_{i_1}} C_{i_2 -1}, \overline{r_{i_1}}C_{i_2}) \cdots \varepsilon( 
\overline{r_{i_1} \cdots r_{i_{k-1}}} C_{i_k -1}, 
\overline{r_{i_1} \cdots r_{i_{k-1}}} C_{i_k})$
if $I \neq \{\}$.
\end{theorem}

\section{Simplifying the Sign Formulas $\varepsilon$}
\begin{notation}
We fix our notation for this section.
\begin{itemize}
\item[-] $\gamma \in \Delta^+( \g,\h)$ 
\item[-] $N \in \bbZ^+$
\item[-] $w \in W$
\item[-] $\alpha$ is a simple root with corresponding simple reflection $s$
\item[-] $x \in W$ satisfies $xs > x$.
\end{itemize}
\end{notation}

In this section, we compute $\varepsilon$ in the cases 
which appear in the recursion formulas for computing signed Kazhdan-Lusztig 
polynomials from \cite{Y}.

We begin by showing that the second and third terms in the expression in 
Theorem \ref{wgammadefn} for 
$\varepsilon( H_{x\alpha,N}, xs )$ are $1$.
This requires a careful study of $w_\gamma$ and $w_{x\alpha}$ which was 
defined in \ref{wgammadefn}.

Recall that $\Delta(w) = \lbrace \alpha \in \Delta^+ | w\alpha < 0 
\rbrace$.  If $w=s_{i_1}s_{i_2} \cdots s_{i_k}$ is a reduced expression, 
then $\Delta( w^{-1} ) = \lbrace \alpha_{i_1}, s_{i_1}\alpha_{i_2}, \ldots, 
s_{i_1}s_{i_2} \cdots s_{i_{k-1}} \alpha_{i_k} \rbrace$.
Note that $w_\gamma$ depends on the choice of expression $\gamma = 
s_{i_1}s_{i_2} \cdots s_{i_{k-1}} \alpha_{i_k}$, but for all choices, 
$\gamma \in \Delta( w_\gamma^{-1} )$.  We need the following 
technical lemmas:
\begin{lemma} \label{ContainmentLemma}
There exists a choice of $w_{x\alpha}$ such that $\Delta(
w_{x\alpha}^{-1} ) \subset \Delta( (xs)^{-1} )$.
\begin{proof}
The proof is by double induction on the length of $x$ and the height of 
$x\alpha$. \\
First induction:  on the length of $x$. \\
First base case:  $\ell(x) = 0$.  If $x = 1$, then $w_{x\alpha} = s_\alpha$ and 
$\Delta( w_{x\alpha}^{-1} ) = \lbrace \alpha \rbrace = \Delta( xs )$.
(Note that the only possible height for $x\alpha$ is $1$.)\\
First inductive hypothesis:  Assume for all $x$ of length less than or equal 
to $k$ and $x < xs$ that there exists $w_{x\alpha}$ such that $\Delta( 
(xs)^{-1} ) \supset \Delta( w_{x \alpha}^{-1})$. \\
Induction step:
We prove this by induction on $ht( x\alpha )$. \\
Second base case:
If the height of $x \alpha$ is $1$, then $w_{x\alpha} =
s_{x \alpha}$ has length $1$ and $\Delta( s_{x\alpha} ) = \lbrace x
\alpha \rbrace$.  Since $sx^{-1}x\alpha = s \alpha = - \alpha$, we
have $x\alpha \in \Delta( (xs)^{-1} )$. \\
Consider $x'$ of length $k+1$, $x'\alpha > 0$.  If $ht( x'\alpha ) > 1$, then 
there exists a simple reflection $t=s_\mu$ such that $x'=tx$ for some $x$ of 
length $k$ and $ht( x'\alpha) = ht( tx\alpha ) > ht( x\alpha ) > 0$.  By our 
inductive hypothesis, there exists $w_{x\alpha}$ with $\Delta( 
w_{x\alpha}^{-1}) \subset \Delta( (xs)^{-1} )$.  Since $tx\alpha > 0$, we see 
that $txs > tx$.  Since 
$tx > x$, we have $txs > xs$.  Therefore
$$ \Delta( (txs)^{-1} ) = \Delta( sx^{-1}t ) = t\Delta( (xs)^{-1} ) \cup 
\lbrace \mu \rbrace.$$
Choose $w_{tx\alpha} = tw_{x\alpha}$.  From Lemma 5.3.3 of \cite{Y3}, $tw_{x
\alpha} > w_{x\alpha}$ and therefore
$$ \Delta( w_{tx\alpha}^{-1}) = \Delta( (tw_{x\alpha})^{-1}) = t\Delta( 
w_{x\alpha}^{-1} ) \cup \lbrace \mu \rbrace.$$
It folows that $\Delta( w_{tx\alpha}^{-1}) \subset \Delta( (txs)^{-1} )$.
\end{proof}
\end{lemma}

\begin{lemma} \label{PositivityLemma}
If $\beta \in \Delta( w_\gamma^{-1} )$, then $(\beta, \gamma ) > 0$.
In particular, $\beta$ and $\gamma$ are not orthogonal.
\begin{proof}
If $s_{i_1} \cdots s_{i_k}$ is a reduced expression for $w_{\gamma}$ 
given by Lemma 5.3.3 of \cite{Y3}, then $\beta = s_{i_1}s_{i_2}\cdots 
s_{i_{j - 1}} \alpha_{i_j}$ for some $1 \leq j < k$ and $\gamma = 
s_{i_1}s_{i_2} \cdots s_{i_{k-1}} 
\alpha_{i_k}$.  Thus $(\beta, \gamma ) = ( \alpha_{i_j}, s_{i_j}s_{i_{j+1}} 
\cdots s_{i_{k-1}} \alpha_{i_k} ) > 0$ since 
$ht(  s_{i_j}s_{i_{j+1}} \cdots s_{i_{k-1}} \alpha_{i_k} ) > 
ht(  s_{i_{j+1}}s_{i_{j+2}} \cdots s_{i_{k-1}} \alpha_{i_k} )$ (recall that 
$w_{\gamma}$ was defined to have this height property).
\end{proof}
\end{lemma}

\begin{definition}
Let $\cS_{w_{x\alpha}}^2 = \lbrace \beta \in \Delta(w_{x\alpha}^{-1} ) \, : 
\, |\beta| = |x\alpha|, \beta \neq x\alpha, \text{ and } \beta, s_\beta x\alpha 
\in \Delta( (xs)^{-1} ) \rbrace$.  (See the second term in the formula 
for $\varepsilon$.)
\end{definition}

\begin{lemma}\label{PairingLemma2}
Suppose $w_{x\alpha}$ is constructed as in Lemma \ref{ContainmentLemma}.  
Then $\cS_{w_{x\alpha}}^2  = \{\}$.
\begin{proof}
Suppose $\beta \in \cS_{w_{x\alpha}}^2$.  Since $\beta$ and $x\alpha$ are 
the same length, they generate a type $A_2$ root system.  Since 
$(\beta,x\alpha)>0$ by Lemma \ref{PositivityLemma}, therefore $s_\beta x\alpha 
= x\alpha - \beta$.  To belong to $\cS_{w_{x\alpha}}^2$,  
$\beta \in \Delta( (xs)^{-1}) \Rightarrow sx^{-1} \beta < 0$.  However, 
$sx^{-1} s_\beta x \alpha = sx^{-1}( x\alpha - \beta ) = -\alpha -sx^{-1}\beta$
which must be positive since $\alpha$ is simple and $sx^{-1} \beta < 0$.
This contradicts the condition $s_\beta x\alpha \in \Delta( (xs)^{-1})$, so 
$\beta \not \in \cS_{w_{x\alpha}}^2$.
\end{proof}
\end{lemma}

\begin{definition}
Let $\cS_{w_{x\alpha}}^3 = \lbrace \beta \in \Delta(w_{x\alpha}^{-1} ) \, : 
\, |\beta| \neq |x\alpha| \text{ and } \beta, -s_\beta s_{x\alpha} \beta 
\in \Delta( (xs)^{-1} ) \rbrace$.  (See the third term in the formula 
for $\varepsilon$.)
\end{definition}

\begin{lemma} \label{PairingLemma3}
Suppose $w_{x\alpha}$ is constructed as in Lemma \ref{ContainmentLemma} and 
that $x\alpha$ does not form a type $G_2$ root system with any other roots.  
Then $\cS_{w_{x\alpha}}^3 = \{\}$.
\begin{proof}
Suppose $\beta \in \cS_{w_{x\alpha}}^3$.  Since $x\alpha$ and $\beta$ generate 
a type $B_2$ root system with $(x\alpha, \beta) > 0$, therefore  either
$-s_\beta s_{x\alpha} \beta = x\alpha - \beta$ if $x\alpha$ is long, or 
$-s_\beta s_{x\alpha} \beta = 2x\alpha - \beta$ if $x\alpha$ is short.
In the former case, arguing as in the proof of the previous lemma, we may 
see that it is impossible for both $\beta, -s_\beta s_{x\alpha} \beta \in 
\Delta( (xs)^{-1})$ to be simultaneously satisfied, contradicting $\beta \in 
\cS_{w_{x\alpha}}^3$.
In the latter case, observe that $\beta,  -s_\beta s_{x\alpha} \beta \in 
\Delta( (xs)^{-1})$ yield the formulas $sx^{-1} \beta < 0$ and 
$sx^{-1}( 2x\alpha - \beta ) = -2 \alpha - sx^{-1} \beta < 0$.  Noting that 
$\alpha$ is short and simple and $-sx^{-1} \beta$ is long and positive, using 
standard constructions of root systems such as those on p. 64 of \cite{H2} it 
is apparent that a positive long root minus twice a short simple root 
cannot give a negative root--contradiction.  Therefore $\cS_{w_{x\alpha}}^3 = 
\{\}$.
\end{proof}
\end{lemma}

We arrive at the following:
\begin{theorem} \label{EpsilonValues}
Let $x \in W$ and let $\alpha$ be a simple root with simple reflection $s$ 
such that $xs > x$.  Then:
\begin{enumerate}
\item If $H_{\alpha,N}$ intersects $w \cC_0$, then 
	$\varepsilon( H_{\alpha,N}, w ) = 1$ if $\alpha$ is compact, and 
	$(-1)^N$ if $\alpha$ is noncompact.
\item $\varepsilon( H_{x\alpha,N}, xs ) = (-1)^{\epsilon(N x\alpha)}$.
\end{enumerate}
\begin{proof}
(1):  This is just Lemma 5.2.17 of \cite{Y3}. \\
(2):  First, we see from Theorem \ref{wgammadefn} that the result holds for 
$\alpha$ which forms a type $G_2$ root system with other roots, so it suffices 
to prove the theorem for other settings.
From our previous two lemmas, it suffices to prove that 
if $w_{x\alpha} = s_{i_1} \cdots s_{i_k}$,  then
$\#\{ \text{noncompact } \alpha_{i_j} : | \alpha_{i_j}| \geq |x\alpha| \} 
\equiv \epsilon( x\alpha ) \pmod{2}$.
We prove this result by induction on $k$.  Observe first that the set we count 
makes no reference to the Weyl chamber.  Clearly if $k=1$, then $x\alpha$ 
is simple and the result follows immediately.  Otherwise, suppose $k>1$ and
let $\beta = s_{i_2} \cdots s_{i_{k-1}} \alpha_{i_k}$.  We see that we may 
select $w_\beta = s_{i_2} s_{i_3} \cdots s_{i_k}$.  We may assume by induction 
that 
$\#\{ \text{noncompact } \alpha_{i_j} : | \alpha_{i_j}| \geq |\beta|, 2 \leq j 
\leq k \} = \varepsilon( \beta )$.  Observe that $|\beta| = |x\alpha| = 
|\alpha_{i_k}|$.  If $|\alpha_{i_1}| < |x\alpha| = |\beta|$, then 
$2\frac{(\beta,\alpha_{i_1})}{(\alpha_{i_1},\alpha_{i_1})}$ is even (recall 
we already settled the case of type $G_2$ roots), so 
$\epsilon( x\alpha ) =  \epsilon( s_{i_1} \beta ) = \epsilon( \beta )$.
The sets we count for $x\alpha$ and for $\beta$ are the same, so the result 
holds for $x\alpha$.  If $|\alpha_{i_1}| \geq |x\alpha| = |\beta|$, then 
$x \alpha = s_{i_1} \beta = \beta + \alpha_{i_1}$, from which the result 
follows immediately for $x\alpha$.
\end{proof}
\end{theorem}

\section{Comparing Signed Kazhdan-Lusztig Polynomials and Classical 
Kazhdan-Lusztig Polynomials}
We will use the following notation in this section:
\begin{notation}
\begin{itemize}
\item[-] For $\lambda \in \h^*$, $W_\lambda$ is the integral Weyl group and 
	$w_\lambda^0$ its long element
\item[-] $\Delta_\lambda := \{ \alpha \in \Delta( \g, \h ): (\lambda,
\alpha^\vee) \in \bbZ \}$.  It is a root system with Weyl group $W_\lambda$.
\item[-] $\Pi_\lambda$ is a set of simple roots for $\Delta_\lambda$ determined 
by $\rho$
\end{itemize}
\end{notation}

We begin by recalling the relationship between classical Kazhdan-Lusztig 
polynomials and the Jantzen filtration.

\begin{definition} \label{JantzenFiltration}
Given an analytic family of invariant Hermitian forms $\left< \cdot, \cdot 
\right>_t$ on a finite-dimensional vector space $V$, where $t \in (-\delta,
\delta)$ and the forms are non-degenerate for $t \neq 0$, the Jantzen 
filtration is defined to be 
$$V=V^{\left< 0 \right>} \supset V^{\left< 1 \right>} \supset \cdots
\supset V^{\left< N \right>} = \{ 0 \}$$
where $v \in V^{\left < n \right>}$ if there exists an analytic map 
$\gamma_v : (-\epsilon, \epsilon) \to V$ for some small $\varepsilon > 0$
such that:
\begin{enumerate}
\item $\gamma_v(0) = v$, and
\item for every $u \in V$, as $t$ approaches $0$, 
$\left< \gamma_v(t), u \right>_t$ vanishes at least to order $n$.
\end{enumerate}
There is a natural invariant Hermitian form on $V^{\left< j \right>}$ with 
radical $V^{\left< j+1 \right>}$:
$$\left < u, v \right >^j = \lim_{t \to 0^+} \frac{1}{t^j}
\left< \gamma_u(t), \gamma_v(t) \right>_t \qquad \forall \, u, v \in V$$
which descends naturally to a non-degenerate invariant Hermitian form 
$\left< \cdot, \cdot \right>_j$ on $V^{\left< j \right>} / V^{\left< j+1 
\right>}$.
\end{definition}

\begin{lemma} \label{ChangeLemma}
(\cite{V}, Proposition 3.3)  Using the notation of the previous
definition, let $(p_j, q_j)$ be the signature of $\left< \cdot,\cdot \right>_j$.
Then:
\begin{eqnarray*}
\text{For small } t > 0, \text{ the signature is } && \left( \sum_j p_j, \sum_j 
q_j \right )  \\
\text{For small } t < 0, \text{ the signature is } && \left( \sum_{j
\text{ even}} p_j + \sum_{j \text{ odd}} q_j, \sum_{j \text{ odd}} p_j +
\sum_{j \text{ even}} q_j \right)
\end{eqnarray*}
\end{lemma}

Since Verma modules $M(\lambda)$ may be viewed as all being realized on the 
same vector space $U(\mfn^-)$ and the weight space decomposition is an 
orthogonal decomposition under the Shapovalov form, analytic paths in the real 
subspace of $\h^*$ of imaginary weights and the Shapovalov forms on the 
corresponding Verma modules give rise to the Jantzen filtration on a given 
Verma module.  Let $\lambda$ be antidominant and let $x \in W_\lambda$, the 
integral Weyl group.  We consider Verma modules  of the form $M( x \lambda )$. 
It is well-known that 
the $j^{\text th}$ level of the Jantzen filtration of $M(x \lambda )$ does not 
depend on the choice of analytic path (proved by Barbasch in \cite{Ba}).  
Furthermore, the $j^{\text th}$ level of the Jantzen filtration, 
 $M(x\lambda)_{\left< j \right>} := M(x\lambda)^{\left<j \right>} / 
 M(x\lambda)^{\left< j+1 \right>}$ is semisimple and a direct sum of modules 
of the form $L( y \lambda )$ where $y \in W_\lambda$ and $y \leq x$.  The 
Jantzen Conjecture states that the multiplicity of any particular irreducible 
highest weight module in the $j^{\text{th}}$ level of the Jantzen filtration 
may be determined by classical Kazhdan-Lusztig polynomials:

\begin{theorem} (\cite{BB}) Jantzen's Conjecture:
Let $\lambda$ be antidominant, $x, y \in W_\lambda$. Then:
$$[ M( x \lambda )_{\left< j \right>} : L( y \lambda )] = \text{coefficient of }
q^{(\ell(x) - \ell(y) - j)/2} \text{ in } P_{w_\lambda^0 x, w_\lambda^0 y}(q).$$
\end{theorem}

While the vectors in the $j^{\text th}$ level of the Jantzen filtration of 
$M(x \lambda )$ are independent of the choice of analytic path, the signature 
of $\left< \cdot, \cdot \right>_j$ on $M(x\lambda)_{\left< j\right>}$ is not.  
For example, combining Lemma \ref{WallCrossing} and Lemma \ref{ChangeLemma}, 
one observes that the signature of $\left< \cdot, \cdot \right>_1$ depends on 
the direction of the analytic path.

For the purpose of studying signatures, rather than recording multiplicities 
in the $j^{\text{th}}$ level of the Jantzen filtration, contributions by all 
$L(y \lambda )$'s to the 
signature of $\left< \cdot, \cdot \right>_j$ are recorded for a particular 
filtration direction in signed Kazhdan-Lusztig polynomials:
\begin{definition} (\cite{Y})
Let  $\lambda$ be antidominant and let $x, y \in W_\lambda$.  Consider the 
invariant Hermitian forms $\left< \cdot, \cdot \right>_j$ on the various levels 
of the Jantzen filtration of $M(x \lambda )$ arising from an analytic path 
whose direction as $t\to 0^+$ is $\delta$.  If
$$ ch_s \left< \cdot, \cdot \right>_j = \sum_{y \leq x} 
a^{\lambda,\delta}_{w_\lambda^0 x, w_\lambda^0 y, j} ch_s L( y \lambda )$$
where by $ch_s L( y \lambda )$ we mean the signature of the Shapovalov form, 
then the value of $a^{\lambda,\delta}_{w_\lambda^0 x, w_\lambda^0 y, j}$
is the same for all $\delta$ in the interior of the same Weyl chamber.  We 
use the notation $a^{\lambda,\delta}_{w_\lambda^0 x, w_\lambda^0 y, j}$ and 
$a^{\lambda,w}_{w_\lambda^0 x, w_\lambda^0 y, j}$ where $w \in W$ with 
$\delta \in w \mathfrak{C}_0$ interchangeably without further comment. 
Signed Kazhdan-Lusztig polynomials are defined by:
$$ P^{\lambda,w}_{w_\lambda^0 x, w_\lambda^0 y} ( q ) := \sum_{j \geq 0}
a^{\lambda,w}_{w_\lambda^0 x, w_\lambda^0 y} q^{\frac{\ell(x) - \ell(y) 
- j}{2}}.$$
\end{definition}
Note that for small $t > 0$, recalling that $-\rho \in \mathfrak{C}_0$, we 
have by Lemma \ref{ChangeLemma}:
$$ e^{w\rho t} ch_s M( x \lambda + w(-\rho)t ) = \sum_{y \leq x} P^{\lambda, 
w}_{w_\lambda^0 x, w_\lambda^0 y}(1) ch_s L( y \lambda ).$$
The left side is known by work in \cite{Y3}.  We would like a formula for 
$ch_s \, L(x \lambda )$, which requires inversion.
In \cite{Y2}, we will show a simple inversion formula which expresses 
$ch_s L(x \lambda )$ as a linear combination of 
$ch_s M(y \lambda + w_\lambda^0(-\rho)t )$. It vastly improves 
the inversion formula found in \cite{Y}.

The simple inversion formula is related to the main theorem of this paper, 
which we now proceed to state and prove:
\begin{theorem}
Let $\lambda$ be antidominant, and let $x, y \in W_\lambda$.  Then 
signed Kazhdan-Lusztig polynomials are related to classical 
Kazhdan-Lusztig polynomials by:
$$ P^{\lambda, w_\lambda^0}_{x,y}(q) = (-1)^{\epsilon(x \lambda - y \lambda )}
P_{x,y}(-q)$$
where $\epsilon$ is the $\mathbb{Z}_2$-grading on the (imaginary) root lattice.
Specifically, $\epsilon( \mu )$ is the parity of the number of non-compact 
roots in an expression for $\mu$ as an integral linear combination of roots.
\end{theorem}
\begin{proof}
We prove this theorem by induction and by comparing recursive formulas for 
computing classical and signed Kazhdan-Lusztig polynomials after substituting 
appropriate simplifications determined in this paper.

Classical Kazhdan-Lusztig polynomials may be computed using $P_{x,x}(q) = 1$,
$P_{x,y}(q) = 0$ if $x \not \leq y$, and the following recursive formulas 
where $s = s_\alpha$ where $\alpha \in \Pi_\lambda$:
\begin{itemize}
\item[a)] If $ys > y$ and $xs > x \geq y$ then:
 $$P_{w_\lambda^0 x,w_\lambda^0 y} = P_{w_\lambda^0 xs, w_\lambda^0 y}$$

\item[a')] If $sy > y$ and $sx > x \geq y$ then:
$$P_{w_\lambda^0 x,w_\lambda^0 y} = P_{w_\lambda^0 sx, w_\lambda^0 y}$$

\item[b)] If $y > ys$ and $x < xs$ then:
$$ P_{w_\lambda^0 xs,w_\lambda^0 y} + q P_{w_\lambda^0 x, w_\lambda^0 y}
= \sum_{z \in W_\lambda | z < zs} a_{w_\lambda^0 z, w_\lambda^0 y, 1}  \,
q^{\frac{\ell(z)-\ell(y)+1}{2}} P_{w_\lambda^0 x, w_\lambda^0 z} + 
P_{w_\lambda^0 x, w_\lambda^0 ys}.$$
\end{itemize}

Now we study recursive formulas for signed Kazhdan-Lusztig polynomials from 
\cite{Y}.  First, we must note some errata:
\begin{itemize}
\item[-]  The formula before Proposition 4.6.6 uses incorrect notation.
It should say:
\begin{eqnarray*}
ch_s U_\alpha \overline{M(z \lambda + \delta t)_0} &=& \sgn( \bar{c}_{zs}'' 
\bar{c}_{z}')
ch_s L(zs \lambda ) \\
&& + \sgn( \bar{c}_z''(\delta, z\alpha^\vee) \bar{c}_z')
\sum_{y \in W_\lambda | y > ys} a^{\lambda,w}_{w_\lambda^0 z, w_\lambda^0 y,
1} ch_s L(y \lambda )
\end{eqnarray*}
\item[-] Substituting the above into the $j^{\text{th}}$ level formula from 
Proposition 4.6.5 of \cite{Y}, the sign of the first term on the right hand 
side gives rise to the coefficient for the final term on the right hand 
side of Proposition 4.6.6.  Proposition 4.6.6 should actually state: \\
If $x, y \in W_\lambda$ are such that $x < xs$ and $y > ys$ and $x > ys$, 
then:
\begin{eqnarray*}
&&\sgn( \bar{c}_{xs}'' \bar{c}_x' ) P^{\lambda,w}_{w_\lambda^0 xs, w_\lambda^0 
y} (q) + \sgn( \bar{c}_x'' (\delta, x \alpha^\vee ) \bar{c}_x') q P^{\lambda,
w}_{w_\lambda^0 x, w_\lambda^0 y} (q) \\
&& = \sum_{z \in W_\lambda | z < zs} \sgn( \bar{c}_z'' (\delta, z \alpha^\vee)
\bar{c}_{z}') a^{\lambda,w}_{w_\lambda^0 z, w_\lambda^0 y,1} q^{\frac{ \ell(x)
- \ell(y) + 1}{2}} P^{\lambda,w}_{w_\lambda^0 x, w_\lambda^0 z} (q) \\
&& + \sgn( \bar{c}_y'' \bar{c}_{ys}') P^{\lambda,w}_{w_\lambda^0 x, w_\lambda^0 
ys} (q).
\end{eqnarray*}
Note that only the sign in the last term on the right side of the formula 
changed.
\item[-]   Due to the above correction, formula b) in Theorem 4.6.10 of 
\cite{Y} becomes:
\begin{eqnarray*}
&&-(-1)^{\epsilon((\lambda,\alpha^\vee)x\alpha)} P^{\lambda,w}_{w_\lambda^0 xs,
w_\lambda^0 y}(q) + \sgn( \delta, x\alpha^\vee ) q P^{\lambda,w}_{w_\lambda^0 x,
w_\lambda^0 y} (q) \\
\,\qquad &&= \sum_{z \in W_\lambda | z < zs} \sgn( \delta, z\alpha^\vee)
a^{\lambda,w}_{w_\lambda^0 z, w_\lambda^0 y, 1} q^{\frac{\ell(z)-\ell(y)+1}{2}}
P^{\lambda,w}_{w_\lambda^0 x, w_\lambda^0 z}(q) - (-1)^{\epsilon((\lambda,
\alpha^\vee) ys\alpha)} P^{\lambda,w}_{w_\lambda^0 x, w_\lambda^0 ys}(q).
\end{eqnarray*}
The final term in the right side of the formula changed to
$- (-1)^{\epsilon( (\lambda,\alpha^\vee)ys\alpha)} P^{\lambda,w_\lambda^0}_{
w_\lambda^0 x, w_\lambda^0 ys}(q)$.  Just as $\sgn( \bar{c}_{xs}'' \bar{c}_x') 
= (-1)^{\epsilon( (\lambda,\alpha^\vee)x\alpha)}$ 	for 
$x < xs$, so must $\sgn(\bar{c}_{y}''\bar{c}_{ys}') = 
(-1)^{\epsilon( (\lambda,\alpha^\vee)ys\alpha)}$ for $ys < y$.
\end{itemize}

Substituting Theorem \ref{EpsilonValues} into Theorem 4.6.10 of \cite{Y} with 
the erratum in the case $w = w_\lambda^0$ and 
using invariance of the inner product on $\h^*$ under the Weyl group, we 
see that signed Kazhdan-Lusztig polynomials may be computed using 
$P_{x,x}^{\lambda,w_\lambda^0}(q) = 1$, $P^{\lambda,w_\lambda^0}_{x,y}(q) = 0$ 
if $x \not \leq y$, and the following recursive formulas where $s = s_\alpha$ 
where $\alpha \in \Pi_\lambda$:
\begin{itemize}
\item[a)]  If $ys < y$ and $xs > x \geq y$ then:
$$P^{\lambda,w_\lambda^0}_{w_\lambda^0 x, w_\lambda^0 y}(q) = 
(-1)^{\epsilon( (\lambda,\alpha^\vee) x \alpha )} P^{\lambda,
w_\lambda^0}_{w_\lambda^0
xs, w_\lambda^0 y} (q) 
= (-1)^{\epsilon((x\lambda, x \alpha^\vee ) x \alpha)} P^{\lambda,w_\lambda^0}_{
w_\lambda^0 xs, w_\lambda^0 y}(q)$$

\item[a')] If $sy > y$ and $sx > x \geq y$ then:
$$P^{\lambda,w_\lambda^0}_{w_\lambda^0 x, w_\lambda^0 y}(q) = (-1)^{\epsilon(
(x\lambda, \alpha^\vee) \alpha)} P^{\lambda,w_\lambda^0}_{w_\lambda^0 sx, 
w_\lambda^0 y}(q)$$

\item[b)] If $y > ys$ and $x < xs$ then:
\begin{eqnarray*}
 - (-1)^{\epsilon(( x \lambda, x \alpha^\vee ) x \alpha)} P^{\lambda,
w_\lambda^0}_{ w_\lambda^0 xs, w_\lambda^0 y}(q) + q P^{\lambda,
w_\lambda^0}_{w_\lambda^0 x, 
w_\lambda^0 y} (q) &=& q \sum_{z \in W_\lambda | z < zs} a^{\lambda,
w_\lambda^0}_{w_\lambda^0 z, w_\lambda^0 y} q^{\frac{\ell(z) - \ell(y) - 1}{2}} 
P^{\lambda,w_\lambda^0}_{ w_\lambda^0 x,
w_\lambda^0 z}(q) \\
&& -(-1)^{\epsilon( (ys \lambda, ys \alpha^\vee) ys\alpha)}
P^{\lambda,w_\lambda^0}_{w_\lambda^0 x, w_\lambda^0 ys}(q).
\end{eqnarray*}
\end{itemize}
Observe that the theorem holds for $x=1$ and for $x=y$.  We now prove our 
theorem by induction on $x$:  that is, if the theorem holds for $x$, then it 
holds for $xs$ and therefore it holds in general.

Formula a):  if by the induction hypothesis
$$P^{\lambda,w_\lambda^0}_{w_\lambda^0 x, w_\lambda^0 y} (q) = (-1)^{\epsilon( 
w_\lambda^0 x \lambda - w_\lambda^0 y \lambda)} P_{w_\lambda^0x, w_\lambda^0 y}
(-q)$$
and $ys < y$ and $xs > x \geq y$, then formula a) gives
$$P^{\lambda,w_\lambda^0}_{w_\lambda^0 xs, w_\lambda^0 y}(q) = (-1)^{\epsilon( 
w_\lambda^0 xs \lambda - w_\lambda^0 y \lambda)} P_{w_\lambda^0 xs, w_\lambda^0 
y}(-q)$$ 
since $xs \lambda - x \lambda = s_{x\alpha} x \lambda - x\lambda = (x\lambda,
x\alpha^\vee) x\alpha$.

Formula a'):  if by the induction hypothesis
$$P^{\lambda,w_\lambda^0}_{w_\lambda^0 x, w_\lambda^0 y} (q) = (-1)^{\epsilon( 
w_\lambda^0 x \lambda - w_\lambda^0 y \lambda)} P_{w_\lambda^0x, w_\lambda^0 y}
(-q)$$ 
and $sy > y$ and $sx > x \geq y$, then
$$P^{\lambda,w_\lambda^0}_{w_\lambda^0 sx, w_\lambda^0 y}(q) = (-1)^{\epsilon( 
w_\lambda^0 sx \lambda - w_\lambda^0 y \lambda)} P_{w_\lambda^0 sx, w_\lambda^0 
y}(-q)$$ 
since 
$s x\lambda - x\lambda = (x\lambda,\alpha^\vee)\alpha$.

Formula b):  Suppose $y > ys$ and $x < xs$.  If by the induction hypothesis:
\begin{itemize}
\item[] $P^{\lambda,w_\lambda^0}_{ w_\lambda^0 x, w_\lambda^0 y}(q) = 
(-1)^{\epsilon( w_\lambda^0 x- w_\lambda^0 y)}P_{w_\lambda^0 x, w_\lambda^0 y}
(-q)$,
\item[] $P^{\lambda,w_\lambda^0}_{w_\lambda^0 x, w_\lambda^0 z}(q) = (-1)^{
\epsilon( w_\lambda^0 x \lambda - w_\lambda^0 z \lambda)} P_{w_\lambda^0 x,
w_\lambda^0 z}(-q)$, 
\item[] $a^{\lambda,w_\lambda^0}_{w_\lambda^0 z, w_\lambda^0 y,1} = 
a_{w_\lambda^0 z, w_\lambda^0 y, 1} (-1)^{\epsilon( w_\lambda^0 z \lambda 
- w_\lambda^0 y\lambda)}
(-1)^{\frac{\ell(z)-\ell(y)-1}{2}}$, and 
\item[] $P^{\lambda,w_\lambda^0}_{w_\lambda^0 x, w_\lambda^0 ys}(q) = 
(-1)^{\epsilon( w_\lambda^0 x - w_\lambda^0 ys)} P_{w_\lambda^0 x, w_\lambda^0 
ys}(-q)$,
\end{itemize}
(recall that we can in fact apply induction to $z$ since $z \leq x$, or else
$P^{\lambda,w_\lambda^0}_{ w_\lambda^0 x, w_\lambda^0 z}(q) = 0)$,
formula b) may be rewritten:
\begin{eqnarray*}
(-1)^{\epsilon(w_\lambda^0(xs \lambda - x \lambda))}&& P^{\lambda,
w_\lambda^0}_{w_\lambda^0 xs,
w_\lambda^0 y}(q) = q (-1)^{\epsilon( w_\lambda^0( x \lambda - y \lambda))} 
P_{w_\lambda^0 x, 
w_\lambda^0 y}(-q) \\
&& - q \sum_{z \in W_\lambda | z < zs} a_{w_\lambda^0 z,
w_\lambda^0 y,1} (-1)^{\epsilon(w_\lambda^0( z\lambda - y \lambda))} 
(-q)^{\frac{\ell(z)-\ell(y)-1}{2}} (-1)^{\epsilon( w_\lambda^0( x \lambda
- z \lambda))} P_{w_\lambda^0 x, w_\lambda^0 z}(-q) \\
&&+(-1)^{\epsilon(w_\lambda^0( y \lambda - ys \lambda))}(-1)^{\epsilon( 
w_\lambda^0 (x \lambda - ys \lambda) )} P_{w_\lambda^0 x, 
w_\lambda^0 ys}(-q).
\end{eqnarray*}
Rearranging, we obtain:
\begin{eqnarray*}
P^{\lambda,w_\lambda^0}_{w_\lambda^0 xs, w_\lambda^0 y}(q) = (-1)^{\epsilon(
w_\lambda^0(xs\lambda-y\lambda))}&&\big[q P_{w_\lambda^0 x, w_\lambda^0 y}
(-q)  )  \\
&& -q \sum_{z \in W_\lambda | z < zs} a_{w_\lambda^0 z, w_\lambda^0 y,1} 
(-q)^{\frac{\ell(z)-\ell(y)-1}{2}} P_{w_\lambda^0, w_\lambda^0 z}(-q) \\
&& + P_{w_\lambda^0 x, w_\lambda^0 ys}(-q) \big].
\end{eqnarray*}
Observe that up to multiplication by $(-1)^{\epsilon(w_\lambda^0(xs\lambda-
y\lambda))}$,
the right hand side is simply the formula for $P_{w_\lambda^0 xs,
w_\lambda^0 y}(q)$ arising from classical formula b) with $-q$ in place of $q$, 
whence 
$$ P^{\lambda,w_\lambda^0}_{w_\lambda^0 xs, w_\lambda^0 y}(q) = 
(-1)^{\epsilon(w_\lambda^0( xs \lambda - y \lambda ))} P_{w_\lambda^0 xs, 
w_\lambda^0 y} (-q)$$
and the theorem holds in general by induction.
\end{proof}

\section{Conclusion}
Although the classification of unitary highest weight modules has been 
solved by work of Enright-Howe-Wallach, it would be interesting to recover the 
classification using signed Kazhdan-Lusztig polynomials and the formulas in 
this paper and in \cite{Y2}.  Cohomological induction applied to highest 
weight modules produces Harish-Chandra modules for which signatures can be 
recorded using signed Kazhdan-Lusztig-Vogan polynomials.  Techniques used to 
identify unitary representations among highest weight modules may very well 
have analogues for Harish-Chandra modules.

It would also be interesting to investigate signed Kazhdan-Lusztig polynomials 
for the non-equal rank case.

\nocite{YC}
\nocite{YE}
\bibliographystyle{alpha}
\bibliography{SimpleSKL}

\begin{thebibliography}{Yee09b}

\bibitem[Bar83]{Ba}
Dan Barbasch.
\newblock Filtrations on {V}erma modules.
\newblock {\em Ann. Sci. \'Ecole Norm. Sup. (4)}, 16(3):489--494, 1983.

\bibitem[BB93]{BB}
A.~Beilinson and J.~Bernstein.
\newblock {\em A proof of {J}antzen Conjectures}, volume~16 of {\em Adv. Soviet
  Math.}
\newblock Amer. Math. Soc., Providence, RI, 1993.

\bibitem[Hum72]{H2}
James~E.\ Humphreys.
\newblock {\em Introduction to Lie Algebra and Representation Theory}.
\newblock Number~9 in Graduate Texts in Mathematics. Springer-Verlag, New York,
  1972.

\bibitem[Vog84]{V}
David~A. Vogan.
\newblock Unitarizability of certain series of representations.
\newblock {\em Annals of Mathematics}, 120:141--187, 1984.

\bibitem[Wal84]{W}
Nolan~R. Wallach.
\newblock On the unitarizability of derived functor modules.
\newblock {\em Invent. Math.}, 78(1):131--141, 1984.

\bibitem[Yee]{Y2}
Wai~Ling Yee.
\newblock Classical and signed {K}azhdan-{L}usztig polynomials: Character
  multiplicity inversion by induction.
\newblock {\em Preprint}.
\newblock arxiv ID.

\bibitem[Yee05]{Y3}
Wai~Ling Yee.
\newblock The signature of the {S}hapovalov form on irreducible {V}erma
  modules.
\newblock {\em Representation Theory}, 9:638--677, 2005.

\bibitem[Yee08]{Y}
Wai~Ling Yee.
\newblock Signatures of invariant {H}ermitian forms on irreducible highest
  weight modules.
\newblock {\em Duke Mathematical Journal}, 142:165--196, 2008.

\bibitem[Yee09a]{YC}
Wai~Ling Yee.
\newblock Signed {K}azhdan-{L}usztig polynomials for compact real forms.
\newblock Technical report, University of Windsor, Windsor, Ontario, 2009.

\bibitem[Yee09b]{YE}
Wai~Ling Yee.
\newblock Signed {K}azhdan-{L}usztig polynomials for compact real forms:
  Erratum.
\newblock Technical report, University of Windsor, Windsor, Ontario, 2009.

\end{thebibliography}
\end{document}